\documentclass[reqno]{amsart}
\usepackage{amsmath,amsfonts,amsthm,amssymb}
\usepackage{graphics}
\usepackage[all,color]{xy}
\usepackage{multicol}
\usepackage{amscd}
\usepackage{graphicx}
\input xy
\xyoption{all}
\xyoption{color}
\usepackage{xcolor}

\xyoption{color}

\theoremstyle{definition}
\newtheorem{thm}{Theorem}[section]

\newtheorem{rmk}[thm]{Remark}

\newtheorem{defn}[thm]{Definition}

\newtheorem{que}[thm]{Question}

\xyoption{knot} \UseComputerModernTips \knottips{FF}
\def\Bracket#1{\mathord{\Big\langle\,~ \raise8pt\xybox{0;/r1.3pc/:#1}\,
~\Big\rangle}}

\begin{document}

\title{Presentation of immersed surface-links by 
marked graph diagrams}


\author{Seiichi Kamada}
\address{Department of Mathematics,  Osaka City University, Sumiyoshi, Osaka 558-8585, Japan}
\email{skamada@sci.osaka-cu.ac.jp} 
\thanks{Supported by JSPS KAKENHI Grant Numbers 26287013 and 15F15319.}

\author{Akio Kawauchi}
\address{Osaka City University Advanced Mathematical Institute, Osaka City University, Sumiyoshi, Osaka 558-8585, Japan}
\email{kawauchi@sci.osaka-cu.ac.jp} 
\thanks{Supported by JSPS KAKENHI Grant Number 24244005.}

\author{Jieon Kim}
\address{Osaka City University Advanced Mathematical Institute, Osaka City University, Sumiyoshi, Osaka 558-8585, Japan}
\email{jieonkim@sci.osaka-cu.ac.jp} 
\thanks{The third author is International Research Fellow
of Japan Society for the Promotion of Science. }

\author{Sang Youl Lee}
\address{Department of Mathematics, Pusan National University, Busan 46241, Republic of Korea}
\email{sangyoul@pusan.ac.kr} 
\thanks{Supported by Basic Science Research Program through the National
Research Foundation of Korea(NRF) funded by the Ministry of Education, Science
and Technology (NRF-2016R1A2B4016029).}

\maketitle

\begin{abstract}
It is well known that 
surface-links in $4$-space can be presented by diagrams on the plane 
of $4$-valent spatial graphs with makers on the vertices, called marked graph diagrams.  
In this paper we extend the method of presenting surface-links by marked graph diagrams to presenting immersed surface-links. 
We also give some moves on marked graph diagrams that preserve the ambient isotopy classes of their presenting immersed surface-links. 
\end{abstract}


\section{Introduction}\label{sect-intr}
\label{intro}

A surface-link, or an {\it embedded} surface-link, is a closed surface embedded in Euclidean $4$-space $\mathbb R^4$.  
An {\it immersed surface-link} is a closed surface immersed in $\mathbb R^4$ such that the multiple points are transverse double points. It is well known that 
surface-links can be presented by diagrams on the plane 
of $4$-valent spatial graphs with makers on the vertices, called marked graph diagrams (cf. \cite{CKS2004, KamBook2017, Kaw, KeKU, Lo, Sw, Yo}).  

In this paper we extend the method of presenting surface-links by marked graph diagrams to presenting immersed surface-links.  
We also give some moves on marked graph diagrams that preserve the ambient isotopy classes of their presenting immersed surface-links, which are extension of moves given by Yoshikawa \cite{Yo} for presentation of embedded surface-links. 


\section{Marked graph diagrams of immersed surface-links}
\label{sect-mgd}

In this section, we introduce a marked graph presentation of immersed  surface-links. First, we recall quickly the notion of marked graph diagrams and links with bands from \cite{Sw, Yo}. 

Let $A$ be the square $\{(x, y)|-1 \leq x, y \leq 1\}$, $X$ be the diagonals in $A$ presented by $x^2 = y^2$, 
and $M_h$ (or $M_v$) be a thick interval in $A$ given by 
$\{(x, y)|-1/2 \leq x \leq 1/2, \, -\delta \leq y \leq \delta \}$ 
(or $\{(x, y)|-1/2 \leq y \leq 1/2, \, -\delta \leq x \leq \delta \}$), where $\delta$ is a small positive number. 

A {\it marked graph} (in $\mathbb R^3$) is a spatial graph $G$ in $\mathbb R^3$ which satisfies the following:
\begin{itemize}
  \item [(1)] $G$ is a finite regular graph with $4$-valent vertices.
  \item [(2)] Each vertex $v$ is rigid; that is, there is a neighborhood $N(v)$ of $v$ which is identified with thickened $A$ 
  such that 
  $v$ corresponds to the origin and the edges restricted to $N(v)$ correspond to $X$. 
  \item [(3)] Each $v$ has a {\it marker}, which is a thick interval in $N(v)$ which corresponds to $M_h$ or $M_v$ under the identification in (2). 
\end{itemize}
An {\it orientation} of a marked graph $G$ is a choice of an orientation for each edge of $G$ 
such that around every vertex $v$, two edges incident to $v$ in a diagonal position are oriented toward $v$ and the other two incident edges are oriented outward. For example, see Fig.~\ref{fig-ori-vert}.  


\begin{figure}[ht]
\begin{center}
\begin{picture}(0,10)
\xy 
(-5,5);(5,-5) **@{-}, 
(5,5);(-5,-5) **@{-}, 
(3,3.2)*{\llcorner}, 
(-3,-3.4)*{\urcorner}, 
(-2.5,2)*{\ulcorner},
(2.5,-2.4)*{\lrcorner}, 
(3,-0.2);(-3,-0.2) **@{-},
(3,0);(-3,0) **@{-}, 
(3,0.2);(-3,0.2) **@{-}
\endxy 
\end{picture}
\vspace{0.5cm} 
\caption{An orientation around a marked vertex}\label{fig-ori-vert}
\end{center}
\end{figure}
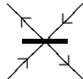

Not every marked graph admits an orientation. 
A marked graph is called {\it orientable} (or {\it non-orientable}) if it admits (or does not admit) an orientation. 
A marked graph depicted in 
Fig.~\ref{fig-nori-mg} is non-orientable. 
An {\it oriented marked graph} is a marked graph equipped with an orientation.  
Two (oriented) marked graphs are said to be {\it equivalent} if they are ambient isotopic in $\mathbb R^3$ with respect to markers as subsets of $\mathbb R^3$ (and the orientations).  

\begin{figure}[ht]
\begin{center}
\resizebox{0.2\textwidth}{!}{%
  \includegraphics{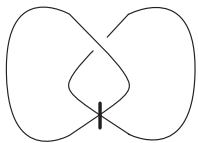}}
\caption{A non-orientable marked graph}\label{fig-nori-mg}
\end{center}
\end{figure}

A {\it banded link} $\mathcal {BL}$ (or a {\it link with bands}) is a pair $(L, \mathcal B)$ of a link $L$ in $\mathbb R^3$ and a set of mutually disjoint bands in $\mathbb R^3$ 
attached to $L$.  It is called {\it oriented} if $L$ is oriented and all bands are oriented coherently with respect to the orientation of $L$. In this case, the link obtained from $L$ by surgery along the bands inherits an orientation, see Fig. \ref{fig-lb}.  
Two (oriented) banded links are {\it equivalent} if there is an ambient isotopy of $\mathbb R^3$ carrying the (oriented) link and (oriented) bands of one to those of the other.

\begin{figure}[ht]
\begin{center}
\resizebox{0.6\textwidth}{!}{%
  \includegraphics{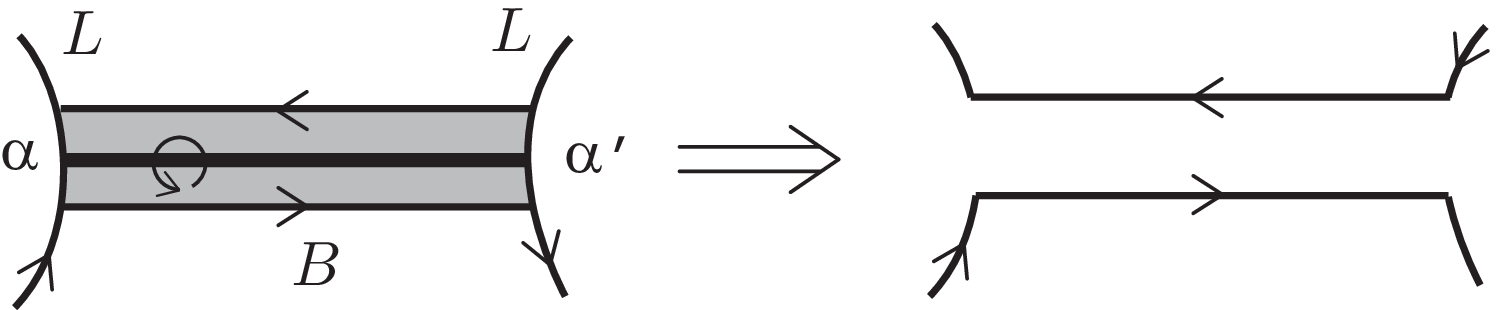}}
\caption{Surgery}\label{fig-lb}
\end{center}
\end{figure}

For a marked graph $G$, we obtain a banded link $(L,\mathcal B) $ by replacing a neighborhood of each $4$-valent vertex with a band such that the core of the band corresponds to the marker as in Fig.~\ref{fig-orbd-2} (b). The banded link is called the {\it banded link associated with $G$} and is denoted by $\mathcal {BL}(G)$. Conversely, a marked graph $G$ is recovered from a banded link $\mathcal {BL}$ by shortening and replacing each band to a $4$-valent vertex as in Fig.~\ref{fig-orbd-2} (a). 
If $G$ is oriented, then $\mathcal {BL}(G)$ is oriented, and vice versa. 

 \begin{figure}[ht]
\begin{center}
\resizebox{0.50\textwidth}{!}{%
  \includegraphics{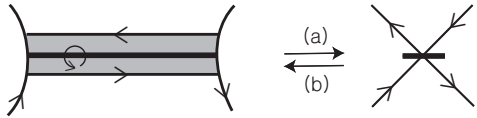} }
\caption{A band and a marked vertex}\label{fig-orbd-2}
\end{center}
\end{figure}

For a banded link $\mathcal {BL}=(L, \mathcal B)$, the {\it lower resolution} $L_-(\mathcal {BL})$ is $L$ and the
{\it upper resolution} $L_+(\mathcal {BL})$ is the surgery result.  For a marked graph $G$, the lower resolution $L_-(G)$ and the upper resolution $L_+(G)$ are defined to be those of the 
banded link $\mathcal {BL}(G)=(L, \mathcal B)$ associated with $G$.

We present a marked graph by a diagram on the plane, which we call a {\it marked graph diagram}, in a usual way in knot theory. 

Let $D$ be a marked graph diagram. We denote by $\mathcal {BL}(D)$, $L_-(D)$, and $L_+(D)$ the banded link, the lower  resolution and the upper resolution of the marked graph presented by $D$.  See Fig.~\ref{spun-mgraph-res}.

A link is called {\it H-trivial} if it is a split union of trivial knots and Hopf links \cite{KamKawamu}. A trivial link is regarded as an H-trivial link without Hopf links.  

\begin{defn} 
A marked graph diagram $D$ (or a marked graph $G$) 
is {\it H-admissible} if both resolutions $L_-(D)$ and $L_+(D)$ (or $L_-(G)$ and $L_+(G)$) 
are H-trivial links.
\end{defn}

\begin{figure}[h]
\begin{center}
\resizebox{0.40\textwidth}{!}{%
  \includegraphics{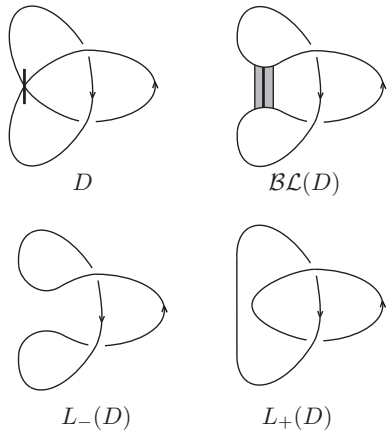} }
\caption{An H-admissible marked graph diagram} 
\label{spun-mgraph-res}
\end{center}
\end{figure}

A marked graph diagram $D$ (or a marked graph $G$) 
is called {\it admissible} if both resolutions $L_-(D)$ and $L_+(D)$ (or $L_-(G)$ and $L_+(G)$) 
are trivial links.  By definition, an admissible marked graph (diagram) is H-admissible.   

Now we discuss a marked graph presentation of an immersed surface-link.  

For a subset $A\subset \mathbb R^3$ and an interval $I\subset\mathbb R,$ let $$AI=\{(x,t)\in\mathbb R^4|x\in A, t\in I\}.$$

Let $D$ be an H-admissible marked graph diagram, and $\mathcal {BL}(D) = (L, \mathcal B)$ the banded link 
associated with $D$.  
Consider a surface $\mathcal{S}^1_{-1}$ in $\mathbb R^3 [-1,1]$ satisfying 
\begin{equation*}
\mathcal S^1_{-1} \cap \mathbb R^3[t]=\left\{%
\begin{array}{ll}
L_+(D)[t] & \hbox{for $0 < t \leq1$,} \\
(L_-(D) \cup |\mathcal B|)[t] & \hbox{for $t = 0$,} \\
L_-(D)[t] & \hbox{for $-1 \leq t < 0$,} \\
\end{array}%
\right.  
\end{equation*}
where $|\mathcal B|$ denotes the union of the bands belonging to $\mathcal B$.  

When $D$ is oriented, we assume that the surface $S^1_{-1}$ is oriented so that the orientation of 
$L_+(D)[1]$ as the boundary of $S^1_{-1}$ coincides with the orientation of 
$L_+(D)$ induced from $D$. 

Let $L$ be an H-trivial link with trivial knot components $O_i$ $(i=1,\ldots,m)$ 
and Hopf link components $H_j$ $(j=1,\ldots,n)$, where $m \geq 0$ and $n \geq 0$. For an interval $[a,b]$, let $L_{\wedge}[a,b]$ be the union of disks 
$\Delta_i$ $(i=1,\ldots,m)$ and $n$ pairs of disks $C_j$ $(j=1,\ldots,n)$ in $\mathbb R^3[a, b]$ such that (1) 
$\partial \Delta_i = O_i[a]$ and $\partial C_j = H_j[a]$, (2) $\Delta_i$ has a unique maximal point, (3) 
each disk of $C_j$ has a unique maximal point, and (4) the two disks of $C_j$ intersect in a point transversely. 
We call $\Delta_i$ $(i=1,\ldots,m)$ a {\it cone system} with base $O_i$ $(i=1,\ldots,m)$ 
and $C_j$ $(j=1,\ldots,n)$ a {\it cone system} with base $H_j$ $(j=1,\ldots,n)$. 

We often assume an additional condition: (5) 
for each cone $C_j$ over $H_j$, the intersection point of the two disks of $C_j$ in condition~(4) is the unique maximal point of each of the disks in condition~(3). 
Similarly, for an H-trivial link $L'$ with trivial knot components $O'_i$ $(i=1,\ldots,m')$ 
and Hopf link components $H'_j$ $(j=1,\ldots,n')$, where $m' \geq 0$ and $n' \geq 0$, 
let $L'_{\vee }[a,b]$ be the disjoint union of 
a cone system $\Delta'_i$ in $\mathbb R^3[a,b]$
with base $O'_i$ in $\mathbb R^3[a]$ $(i=1,\ldots,m')$ 
and a cone system $C'_j$ in $\mathbb R^3[a,b]$ with base $H'_j$ in $\mathbb R^3[a]$ $(i=1,\ldots,n')$, 
where each disk in the cone system has a unique minimal point.

Let $D$ be an H-admissible marked graph diagram.  Consider the union 
$$\mathcal S(D)=L'_\vee[-2,-1]\cup\mathcal S^1_{-1}\cup L_\wedge[1,2],$$ 
which is an immersed surface-link in $\mathbb R^4$,  where $L$ and $L'$ be upper and lower resolutions of $D$, respectively. 

By an argument in \cite{KamKawamu, KSS} it is seen that the ambient isotopy class 
of the immersed surface-link $\mathcal S(D)$ is uniquely determined from $D$.  
We call the immersed surface-link $\mathcal S(D)$ the {\it immersed surface-link constructed from} $D$.

\begin{thm}\label{prop-adm}
Let $\mathcal L$ be an immersed surface-link. There is an H-admissible marked graph diagram $D$ such that $\mathcal L$ is 
ambient isotopic to $\mathcal S(D)$. 
\end{thm}

In the situation of this theorem,  we say that $\mathcal L$ is {\it presented} by $D$.

\begin{proof} The following argument is based on an argument in \cite{KSS} where embedded and oriented surface-links are discussed (cf. \cite{KamBook2017}). 
Let $\mathcal L$ be an immersed surface-link. Let $d_1, \dots, d_n$ be the double points of $\mathcal L$, 
and let 
$N(d_1), \dots, N(d_n)$ be regular neighborhoods of them.  
Moving $\mathcal L$ by an ambient isotopy, we may assume the following conditions:  
\begin{itemize}
\item[(1)] All critical points of $\mathcal L$, except the double points,  with respect to the projection $\mathbb R^4 = \mathbb R^3 \times \mathbb R \to \mathbb R$ are elementary critical points, that are maximal points, saddle points and minimal points. 
\item[(2)]  $\mathcal L$  is in $\mathbb R^3 (-2,2)$. 
\item[(3)] All double points are in $\mathbb R^3 [1]$. 
\item[(4)] For each $i$ $(i=1, \dots, n)$, $N(d_i) = N^3(d_i) [1-\epsilon, 1 + \epsilon]$ for a $3$-disk $N^3(d_i)$, and 
$N(d_i) \cap  \mathcal L $ is the cone of a Hopf link $H_i \subset ({\rm int} N^3(d_i))[1-\epsilon]$ with the cone point $d_i \in \mathbb R^3 [1]$. Here $\epsilon$ is a sufficiently small positive number.  
The $3$-disks $N^3(d_1), \dots, N^3(d_n)$ are mutually disjoint. 
\end{itemize} 

Move double points into $\mathbb R^3[3]$ such that the condition (4) is preserved although the $3$-disk $N^3(d_i)$ may change and the time level of $d_i$ changes from $1$ to $3$, i.e., 
\begin{itemize}
\item[(1)] All critical points of $\mathcal L$, except the double points,  with respect to the projection $\mathbb R^4 = \mathbb R^3 \times \mathbb R \to \mathbb R$ are elementary critical points, that are maximal points, saddle points and minimal points. 
\item[(2)]  $\mathcal L$  is in $\mathbb R^3 (-2,4)$. 
\item[(3)] All double points are in $\mathbb R^3 [3]$. All maximal, saddle and minimal points are in $\mathbb R^3 (-2,2)$
\item[(4)] For each $i$ $(i=1, \dots, n)$, $N(d_i) = N^3(d_i) [3-\epsilon, 3 + \epsilon]$ for a $3$-disk $N^3(d_i)$, and 
$N(d_i) \cap  \mathcal L $ is the cone of a Hopf link $H_i \subset ({\rm int} N^3(d_i))[3-\epsilon]$ with the cone point $d_i \in \mathbb R^3 [3]$. 
The $3$-disks $N^3(d_1), \dots, N^3(d_n)$ are mutually disjoint. 
\end{itemize} 

Let $p_1, \dots, p_m$  be the maximal points of $\mathcal L$, 
$q_1, \dots, q_{m'}$  be the minimal points of $\mathcal L$, 
 and let 
$N(p_1), \dots, N(p_m), N(q_1), \dots, N(q_{m'})$ be regular neighborhoods of them.  
Moving $\mathcal L$ by an ambient isotopy, we may assume the following conditions:  
\begin{itemize}
\item[(1)] All critical points of $\mathcal L$, except the double points,  with respect to the projection $\mathbb R^4 = \mathbb R^3 \times \mathbb R \to \mathbb R$ are elementary critical points, that are maximal points, saddle points and minimal points. 
\item[(2)]  $\mathcal L$  is in $\mathbb R^3 (-4,4)$. 
\item[(3)] All double points and all maximal points are in $\mathbb R^3 [3]$. All minimal points are in $\mathbb R^3 [-3]$. 
All saddle points are in  $\mathbb R^3 (-2, 2)$. 
\item[(4)] For each $i$ $(i=1, \dots, n)$, $N(d_i) = N^3(d_i) [3-\epsilon, 3 + \epsilon]$ for a $3$-disk $N^3(d_i)$, and 
$N(d_i) \cap  \mathcal L $ is the cone of a Hopf link $H_i \subset ({\rm int} N^3(d_i))[3-\epsilon]$ with the cone point $d_i$. 
The $3$-disks $N^3(d_1), \dots, N^3(d_n)$ are mutually disjoint.  
\item[(5)] For each $i$ $(i=1, \dots, m)$, $N(p_i) = N^3(p_i) [3-\epsilon, 3 + \epsilon]$ for a $3$-disk $N^3(p_i)$, and 
$N(p_i) \cap  \mathcal L $ is the cone of a trivial knot $O_i \subset ({\rm int} N^3(p_i))[3-\epsilon]$ with the cone point $p_i$. 
The $3$-disks $N^3(p_1), \dots, N^3(p_m)$ are mutually disjoint, and also disjoint from $N^3(d_1), \dots, N^3(d_n)$. 
\item[(6)] For each $i$ $(i=1, \dots, m')$, $N(q_i) = N^3(q_i) [-3-\epsilon, -3 + \epsilon]$ for a $3$-disk $N^3(q_i)$, and 
$N(q_i) \cap  \mathcal L $ is the cone of a trivial knot $O'_i \subset ({\rm int} N^3(q_i))[-3+ \epsilon]$ with the cone point $q_i$. 
The $3$-disks $N^3(q_1), \dots, N^3(q_{m'})$ are mutually disjoint.  
\end{itemize} 

Finally, applying the argument in \cite{KSS}, we can move all saddle points into the same hyperplane $\mathbb R^3 [0]$. 
Then we see the result. 
\end{proof}

\begin{rmk}
Theorem 1.4 of \cite{KamKawamu} states that any immersed and oriented surface-link is ambient isotopic to an immersed surface-link 
satisfying a certain condition. 
Applying the argument in \cite{KSS}, we can obtain an immersed surface-link required in Theorem~\ref{prop-adm}. 
\end{rmk}

\section{Moves on marked graph diagrams}
 \label{sect-moves}

We discuss moves on marked graph diagrams which preserve the ambient isotopy classes of the immersed surface-links presented by the diagrams.

\begin{figure}[ht]
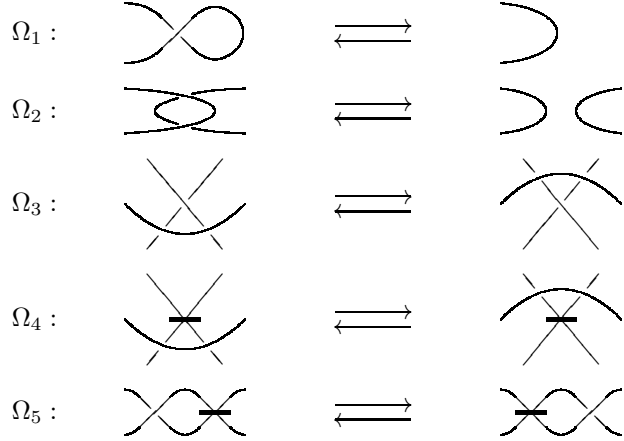

\begin{center}
\centerline{ 
\xy (12,2);(16,6) **@{-}, 
(12,6);(13.5,4.5) **@{-},
(14.5,3.5);(16,2) **@{-}, 
(16,6);(22,6) **\crv{(18,8)&(20,8)},
(16,2);(22,2) **\crv{(18,0)&(20,0)}, (22,6);(22,2) **\crv{(23.5,4)},
(7,8);(12,6) **\crv{(10,8)}, (7,0);(12,2) **\crv{(10,0)}, 
(35,5);(45,5) **@{-} ?>*\dir{>}, (35,3);(45,3) **@{-} ?<*\dir{<},
(57,8);(57,0) **\crv{(67,7)&(67,1)}, (-5,4)*{\Omega_1 :}, (73,4)*{},
\endxy }

\vskip.3cm


\centerline{ \xy (7,7);(7,1)  **\crv{(23,6)&(23,2)}, (16,6.3);(23,7)
**\crv{(19,6.9)}, (16,1.7);(23,1) **\crv{(19,1.1)},
(14,5.7);(14,2.3) **\crv{(8,4)},
(35,5);(45,5) **@{-} ?>*\dir{>}, (35,3);(45,3) **@{-} ?<*\dir{<},
(57,7);(57,1) **\crv{(65,6)&(65,2)}, (73,7);(73,1)
**\crv{(65,6)&(65,2)},
(-5,4)*{\Omega_2 :},
\endxy}

\vskip.3cm

\centerline{ 
\xy (7,6);(23,6) **\crv{(15,-2)}, 
(10,0);(11.5,1.8) **@{-}, 
(17.5,3);(14.5,6.6) **@{-},
(14.5,6.6);(10,12) **@{-}, 
(20,12);(15.5,6.6) **@{-},
(14.5,5.5);(12.5,3) **@{-},
(18.5,1.8);(20,0) **@{-},
(35,7);(45,7) **@{-} ?>*\dir{>}, 
(35,5);(45,5) **@{-} ?<*\dir{<},
(57,6);(73,6) **\crv{(65,14)}, 
(70,12);(68.5,10.2) **@{-}, 
(67.5,9);(65.5,6.5) **@{-}, 
(64.6,5.5);(60,0) **@{-}, 
(62.5,9);(64.4,6.6) **@{-}, 
(64.4,6.6);(70,0) **@{-}, 
(61.5,10.2);(60,12) **@{-},
(-5,6)*{\Omega_3:},
\endxy}

\vskip.3cm

 \centerline{ \xy 
 (7,6);(23,6)  **\crv{(15,-2)}, 
 (10,0);(11.5,1.8) **@{-},
(12.5,3);(20,12) **@{-}, 
(10,12);(17.5,3) **@{-}, 
(18.5,1.8);(20,0) **@{-}, 
(13,6);(17,6) **@{-}, (13,6.1);(17,6.1) **@{-}, (13,5.9);(17,5.9)
**@{-}, (13,6.2);(17,6.2) **@{-}, (13,5.8);(17,5.8) **@{-},
(35,7);(45,7) **@{-} ?>*\dir{>}, 
(35,5);(45,5) **@{-} ?<*\dir{<},
(57,6);(73,6)  **\crv{(65,14)}, 
(70,12);(68.5,10.2) **@{-},
(67.5,9);(60,0) **@{-}, 
(70,0);(62.5,9) **@{-}, 
(61.5,10.2);(60,12) **@{-}, 
(63,6);(67,6) **@{-}, (63,6.1);(67,6.1) **@{-}, (63,5.9);(67,5.9)
**@{-}, (63,6.2);(67,6.2) **@{-}, (63,5.8);(67,5.8) **@{-},
(-5,6)*{\Omega_4:},
\endxy}

\vskip.3cm

 \centerline{ \xy (9,2);(13,6) **@{-}, (9,6);(10.5,4.5) **@{-},
(11.5,3.5);(13,2) **@{-}, (17,2);(21,6) **@{-}, (17,6);(21,2)
**@{-}, (13,6);(17,6) **\crv{(15,8)}, (13,2);(17,2) **\crv{(15,0)},
(7,7);(9,6) **\crv{(8,7)}, (7,1);(9,2) **\crv{(8,1)}, (23,7);(21,6)
**\crv{(22,7)}, (23,1);(21,2) **\crv{(22,1)}, 
(17,4);(21,4) **@{-}, (17,4.1);(21,4.1) **@{-}, (17,3.9);(21,3.9)
**@{-}, (17,4.2);(21,4.2) **@{-}, (17,3.8);(21,3.8) **@{-},
(35,5);(45,5) **@{-} ?>*\dir{>}, (35,3);(45,3) **@{-} ?<*\dir{<},
(59,2);(63,6) **@{-}, (59,6);(63,2) **@{-}, (67,2);(71,6) **@{-},
(67,6);(68.5,4.5) **@{-}, (69.5,3.5);(71,2) **@{-}, (63,6);(67,6)
**\crv{(65,8)}, (63,2);(67,2) **\crv{(65,0)}, (57,7);(59,6)
**\crv{(58,7)}, (57,1);(59,2) **\crv{(58,1)}, (73,7);(71,6)
**\crv{(72,7)}, (73,1);(71,2) **\crv{(72,1)}, 
(63,4);(59,4) **@{-}, (63,4.1);(59,4.1) **@{-}, (63,3.9);(59,3.9)
**@{-}, (63,4.2);(59,4.2) **@{-}, (63,3.8);(59,3.8) **@{-},
 (-5,4)*{\Omega_5:},
\endxy }
\caption{ Moves of Type I}\label{fig-moves-type-I}
\end{center}
\end{figure}

\begin{figure}[ht]
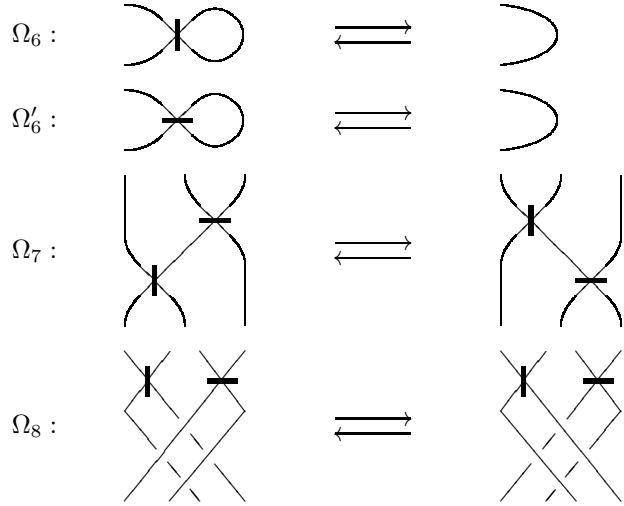

\begin{center}
\centerline{ \xy (12,6);(16,2) **@{-}, (12,2);(16,6) **@{-},
(16,6);(22,6) **\crv{(18,8)&(20,8)}, (16,2);(22,2)
**\crv{(18,0)&(20,0)}, (22,6);(22,2) **\crv{(23.5,4)}, (7,8);(12,6)
**\crv{(10,8)}, (7,0);(12,2) **\crv{(10,0)},
(35,5);(45,5) **@{-} ?>*\dir{>}, (35,3);(45,3) **@{-} ?<*\dir{<},
(57,8);(57,0) **\crv{(67,7)&(67,1)}, (-5,4)*{\Omega_6 :}, (73,4)*{},
(14,6);(14,2) **@{-}, (14.1,6);(14.1,2) **@{-}, (13.9,6);(13.9,2)
**@{-}, (14.2,6);(14.2,2) **@{-}, (13.8,6);(13.8,2) **@{-}, 
\endxy}

\vskip.3cm

\centerline{ \xy (12,6);(16,2) **@{-}, (12,2);(16,6) **@{-},
(16,6);(22,6) **\crv{(18,8)&(20,8)}, (16,2);(22,2)
**\crv{(18,0)&(20,0)}, (22,6);(22,2) **\crv{(23.5,4)}, (7,8);(12,6)
**\crv{(10,8)}, (7,0);(12,2) **\crv{(10,0)},
(35,5);(45,5) **@{-} ?>*\dir{>}, (35,3);(45,3) **@{-} ?<*\dir{<},
(57,8);(57,0) **\crv{(67,7)&(67,1)}, (-5,4)*{\Omega'_6 :},
(73,4)*{}, (12,4);(16,4) **@{-}, (12,4.1);(16,4.1) **@{-},
(12,4.2);(16,4.2) **@{-}, (12,3.9);(16,3.9) **@{-},
(12,3.8);(16,3.8) **@{-}, 
\endxy}

\vskip.3cm

\centerline{ \xy (9,4);(17,12) **@{-}, (9,8);(13,4) **@{-},
(17,12);(21,16) **@{-}, (17,16);(21,12) **@{-}, (7,0);(9,4)
**\crv{(7,2)}, (7,12);(9,8) **\crv{(7,10)}, (15,0);(13,4)
**\crv{(15,2)}, (17,16);(15,20) **\crv{(15,18)}, (21,16);(23,20)
**\crv{(23,18)}, (21,12);(23,8) **\crv{(23,10)}, (7,12);(7,20)
**@{-}, (23,8);(23,0) **@{-},
(11,4);(11,8) **@{-}, 
(10.9,4);(10.9,8) **@{-}, 
(11.1,4);(11.1,8) **@{-}, 
(10.8,4);(10.8,8) **@{-}, 
(11.2,4);(11.2,8) **@{-},
(17,14);(21,14) **@{-}, 
(17,14.1);(21,14.1) **@{-},
(17,13.9);(21,13.9) **@{-}, 
(17,14.2);(21,14.2) **@{-},
(17,13.8);(21,13.8) **@{-},
%
(35,11);(45,11) **@{-} ?>*\dir{>}, (35,9);(45,9) **@{-} ?<*\dir{<},
(71,4);(63,12) **@{-}, (71,8);(67,4) **@{-}, (63,12);(59,16) **@{-},
(63,16);(59,12) **@{-}, (73,0);(71,4) **\crv{(73,2)}, (73,12);(71,8)
**\crv{(73,10)}, (65,0);(67,4) **\crv{(65,2)}, (63,16);(65,20)
**\crv{(65,18)}, (59,16);(57,20) **\crv{(57,18)}, (59,12);(57,8)
**\crv{(57,10)}, (73,12);(73,20) **@{-}, (57,8);(57,0) **@{-},
(61,12);(61,16) **@{-}, 
(61.1,12);(61.1,16) **@{-},
(60.9,12);(60.9,16) **@{-}, 
(61.2,12);(61.2,16) **@{-},
(60.8,12);(60.8,16) **@{-},
%
(67,6);(71,6) **@{-}, 
(67,6.1);(71,6.1) **@{-}, 
(67,5.9);(71,5.9) **@{-}, 
(67,6.2);(71,6.2) **@{-},
(67,5.8);(71,5.8) **@{-},  
(-5,10)*{\Omega_7:}, 
 \endxy}

\vskip.3cm

\centerline{ \xy (7,20);(14.2,11) **@{-}, (15.8,9);(17.4,7) **@{-},
(19,5);(23,0) **@{-}, (13,20);(7,12) **@{-}, (7,12);(11.2,7) **@{-},
(12.7,5.2);(14.4,3.2) **@{-}, (15.7,1.6);(17,0) **@{-},
(17,20);(23,12) **@{-}, (13,0);(23,12) **@{-}, (7,0);(23,20) **@{-},
(10,18);(10,14) **@{-}, (10.1,18);(10.1,14) **@{-},
(9.9,18);(9.9,14) **@{-}, (10.2,18);(10.2,14) **@{-},
(9.8,18);(9.8,14) **@{-}, (18,16);(22,16) **@{-},
(18,16.1);(22,16.1) **@{-}, (18,15.9);(22,15.9) **@{-},
(18,16.2);(22,16.2) **@{-}, (18,15.8);(22,15.8) **@{-},
(35,11);(45,11) **@{-} ?>*\dir{>}, (35,9);(45,9) **@{-} ?<*\dir{<},
(73,20);(65.8,11) **@{-}, (64.2,9);(62.6,7) **@{-}, (61,5);(57,0)
**@{-}, (67,20);(73,12) **@{-}, (73,12);(68.8,7) **@{-},
(67.3,5.2);(65.6,3.2) **@{-}, (64.3,1.6);(63,0) **@{-},
(63,20);(57,12) **@{-}, (67,0);(57,12) **@{-}, (73,0);(57,20)
**@{-},
(60,18);(60,14) **@{-}, (60.1,18);(60.1,14) **@{-},
(59.9,18);(59.9,14) **@{-}, (60.2,18);(60.2,14) **@{-},
(59.8,18);(59.8,14) **@{-}, (68,16);(72,16) **@{-},
(68,16.1);(72,16.1) **@{-}, (68,15.9);(72,15.9) **@{-},
(68,16.2);(72,16.2) **@{-}, (68,15.8);(72,15.8) **@{-},
(-5,10)*{\Omega_8:}, 
\endxy}
\caption{Moves of Type II}\label{fig-moves-type-II}
\end{center}
\end{figure}

The moves depicted in  Figs.~\ref{fig-moves-type-I} and~\ref{fig-moves-type-II} were  
introduced by Yoshikawa \cite{Yo} as moves on marked graph diagrams which do not change the ambient isotopy classes of their presenting surface-links.  The moves and their mirror images are called {\it Yoshikawa moves}. Furthermore, we call the moves in  Fig.~\ref{fig-moves-type-I} 
(Fig.~\ref{fig-moves-type-II}) 
and their mirror images {\it moves of type I} ({\it moves of type II}).  
Moves of type I do not change the ambient isotopy classes of marked graphs in $\mathbb R^3$, and moves of type II do. 
Note that Yoshikawa moves preserve $H$-admissibility and admissibility.   

It is known that two admissible marked graph diagrams present ambient isotopic surface-links if and only if they are related by Yoshikawa moves  (cf. \cite{KeKU, Sw, Yo}).  

Let $D$ be a link diagram of an $H$-trivial link $L$.  A crossing point $p$ of $D$ is an {\it unlinking crossing point} if it is a crossing between two components of the same Hopf link of $L$ and if the crossing change at $p$ makes the Hopf link into a trivial link.  

\begin{defn}
Let $D$ be an $H$-admissible marked graph diagram and let 
$D_-$ and $D_+$ be the diagrams of the lower resolution $L_-(D)$ and the upper resolution $L_+(D)$, respectively.  
A crossing point $p$ of $D$ is a {\it lower singular point} (or an  {\it upper singular point}, resp.) 
if $p$ is an unlinking crossing point of $D_-$ (or $D_+$, resp.). 
\end{defn}

We introduce new moves for $H$-admissible marked graph diagrams. They are 
the moves $\Omega_{9}$, $\Omega_{9}'$ and $\Omega_{10}$ in Fig.~\ref{fig-moves-type-III} and their mirror images, 
which we call {\it moves of type III}. 
Here we assume that the moves of type III are defined only if two diagrams appearing before and after the move are H-admissible.  For example, 
for the move $\Omega_{9}$ (or $\Omega_{9}'$, resp.) in Fig.~\ref{fig-moves-type-III}, we require that the component $l$ in the resolution $L_+(D)$ (or $L_-(D)$, resp.) is trivial and that $p$ is an upper (or lower, resp.) singular point.

\begin{figure}[ht]
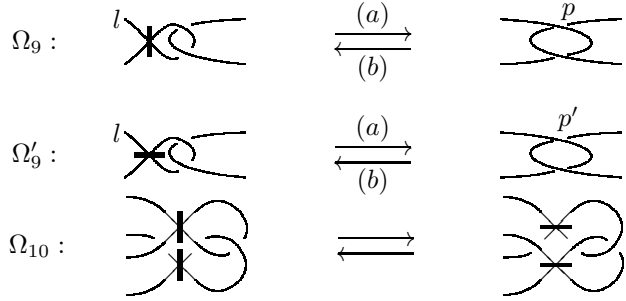

\begin{center}

\centerline{ \xy (57,7);(66,2.3)  **\crv{(71,6)&(71,3)}, (57,1);(64,1.8)  **\crv{(61,1.1)}, (66,6.3);(73,7)**\crv{(69,6.9)}, (73,1);(63.9,5.8)  **\crv{(59,2)&(59,5)}, 
 (40,7.5) *{(a)}, (40,0.5) *{(b)},
(35,5);(45,5) **@{-} ?>*\dir{>}, (35,3);(45,3) **@{-} ?<*\dir{<},
(7,7);(14,1.8)  **\crv{(10,5)&(12,1)},(7,1);(14,6.2)  **\crv{(10,3)&(12,7)}, (14,6.2);(16,3.3)  **\crv{(17,5)}, (16,6.3);(23,7)**\crv{(19,6.9)}, (23,1);(13.9,5.5)  **\crv{(12,2)&(12,4.5)}, 
(10.3,2);(10.3,6) **@{-}, (10.2,2);(10.2,6) **@{-}, (10.1,2);(10.1,6) **@{-}, (10.4,2);(10.4,6) **@{-}, (10.5,2);(10.5,6) **@{-}, (-5,4)*{\Omega_9 :},(66,8) *{p},(6,7) *{l},
\endxy}

\vskip.3cm

\centerline{ \xy (57,7);(66,2.3)  **\crv{(71,6)&(71,3)}, (57,1);(64,1.8)  **\crv{(61,1.1)}, (66,6.3);(73,7)**\crv{(69,6.9)}, (73,1);(63.9,5.8)  **\crv{(59,2)&(59,5)}, 
 (40,7.5) *{(a)}, (40,0.5) *{(b)},
(35,5);(45,5) **@{-} ?>*\dir{>}, (35,3);(45,3) **@{-} ?<*\dir{<},
(7,7);(14,1.8)  **\crv{(10,5)&(12,1)},(7,1);(14,6.2)  **\crv{(10,3)&(12,7)}, (14,6.2);(16,3.3)  **\crv{(17,5)}, (16,6.3);(23,7)**\crv{(19,6.9)}, (23,1);(13.9,5.5)  **\crv{(12,2)&(12,4.5)}, 
(8.3,4);(12.3,4) **@{-}, (12.3,4.1);(8.3,4.1) **@{-}, (12.3,3.9);(8.3,3.9)**@{-}, (12.3,4.2);(8.3,4.2) **@{-}, (12.3,3.8);(8.3,3.8)**@{-}, 
(-5,4)*{\Omega_9' :},(66,8.8) *{p'},(6,7) *{l},
\endxy}

\centerline{ \xy (12,6);(16,2) **@{-}, (12,2);(16,6) **@{-},
(16,6);(22,6) **\crv{(18,8)&(20,8)}, (16,2);(20.5,0.5)
**\crv{(18,0)&(20,0)}, (22,6);(22.5,2.5) **\crv{(23.5,4)}, (7,8);(12,6)
**\crv{(10,8)}, (7,0);(12,2) **\crv{(10,0)}, 
(35,2.5);(45,2.5) **@{-} ?>*\dir{>}, (35,0.5);(45,0.5) **@{-} ?<*\dir{<},
(-5,1.5)*{\Omega_{10} :}, (73,4)*{},
(14,6);(14,2) **@{-}, (14.1,6);(14.1,2) **@{-}, (13.9,6);(13.9,2)
**@{-}, (14.2,6);(14.2,2) **@{-}, (13.8,6);(13.8,2) **@{-}, 
(12.5,0.5);(16,-3) **@{-}, (12,-3);(15.5,0.5) **@{-},
(17.5,2.5);(22,1) **\crv{(18,3)&(21,3)}, (16,-3);(22,-3)
**\crv{(18,-5)&(20,-5)}, (22,1);(22,-3) **\crv{(23.5,-1)}, (7,3);(10.5,2.5)
**\crv{(10,3)}, (7,-5);(12,-3) **\crv{(10,-5)}, 
(73,-1)*{},
(14,1);(14,-3) **@{-}, (14.1,1);(14.1,-3) **@{-}, (13.9,1);(13.9,-3)
**@{-}, (14.2,1);(14.2,-3) **@{-}, (13.8,1);(13.8,-3) **@{-}, 
(62,6);(65.5,2.5) **@{-}, (62.5,2.5);(66,6) **@{-},
(66,6);(72,6) **\crv{(68,8)&(70,8)}, (67.5,0.5);(72,2)
**\crv{(68,0)&(71,0)}, (72,6);(72,2) **\crv{(73.5,4)}, (57,8);(62,6)
**\crv{(60,8)}, (57,0);(60.6,0.6) **\crv{(60,0)},
(62,4);(66,4) **@{-}, (62,4.1);(66,4.1) **@{-}, (62,3.9);(66,3.9)
**@{-}, 
(62,1);(66,-3) **@{-}, (62,-3);(66,1) **@{-},
(66,1);(70.7,2.3) **\crv{(68,3)&(70,3)}, (66,-3);(72,-3)
**\crv{(68,-5)&(70,-5)}, (72.5,0.5);(72,-3) **\crv{(73.5,-1)}, (57,3);(62,1)
**\crv{(60,3)}, (57,-5);(62,-3) **\crv{(60,-5)},
(62,-1);(66,-1) **@{-}, (62,-1.1);(66,-1.1) **@{-}, (62,-0.9);(66,-0.9)
**@{-},
\endxy}

\vskip.3cm

\caption{Moves of Type III: $\Omega_9,\Omega_9'$ and $\Omega_{10}$}\label{fig-moves-type-III}
\end{center}
\end{figure}

\begin{defn} 
The {\it generalized Yoshikawa moves} are Yoshikawa moves (moves of type I and II) and moves of type III introduced above.  
Two marked graph diagrams are {\it stably equivalent} if they are related by a finite sequence of generalized Yoshikawa moves.
\end{defn}

\begin{thm}\label{thm-equiv}
Let $\mathcal L$ and $\mathcal L'$ be immersed surface-links presented by marked graph diagrams $D$ and $D'$, respectively. If $D$ and $D'$ are stably equivalent, then $\mathcal L$ and $\mathcal L'$ are ambient isotopic. 
\end{thm}

\begin{proof}
It suffices to show that  $\mathcal L$ and $\mathcal L'$ are ambient isotopic when $D'$ is obtained from $D$  by  a  move of $\Omega_9,$ $\Omega_9'$ or $\Omega_{10}.$ The moves $\Omega_9$ and  $\Omega_9'$ correspond to a creation or removal of a saddle point, and the move $\Omega_{10}$ corresponds to a change the level of double point singularity. See Fig.~\ref{fig-nm4}, which shows partial pictures of broken surface diagrams in $3$-space in the sense of \cite{CS1}.  
(In \cite{CS1}, embedded surfaces are discussed. However, broken surface diagrams are considered for immersed surface-links and it is true that if two broken surface diagrams are ambient isotopic in $3$-space then the immersed surface-links are ambient isotopic in $4$-space.) 
Since the moves $\Omega_9$, $\Omega_9'$ and  $\Omega_{10}$ do not change the ambient isotopy classes of broken surface diagrams in $3$-space, we see that $\mathcal L$ and $\mathcal L'$ are ambient isotopic. 
\end{proof} 

\begin{figure}[ht]
\begin{center}
\resizebox{0.8\textwidth}{!}{%
  \includegraphics[width=2cm]{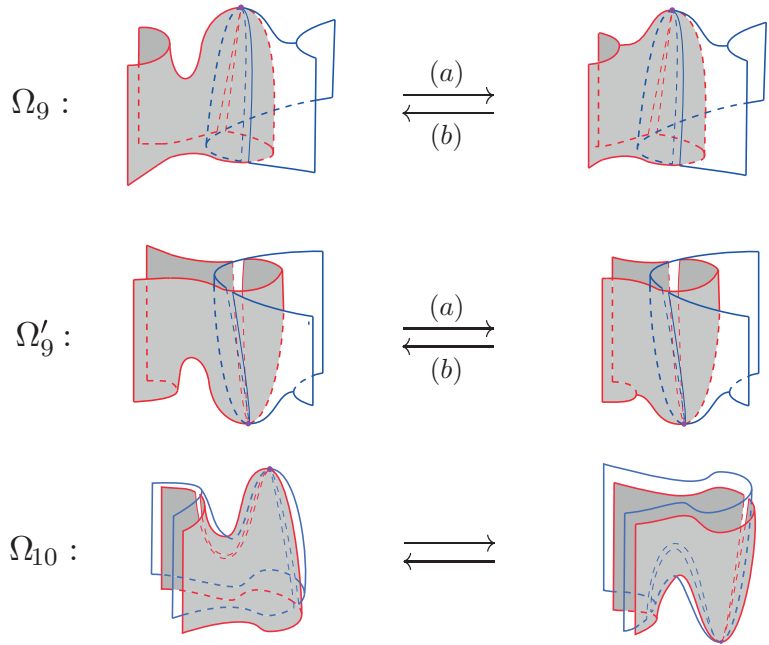}}
\caption{Immersed surface-links presented by $\Omega_9$, $\Omega_9'$, and $\Omega_{10}$}\label{fig-nm4}
\end{center}
\end{figure}

Let $\Omega_9^*$ and $\Omega_9'^*$ be the moves depicted in Fig.~\ref{fig-omoves} or their mirror images.  
They are equivalent to $\Omega_9$ and $\Omega_9'$ modulo Yoshikawa moves (of type I) as shown in Fig.~\ref{fig-m9a}.

\begin{figure}[ht]
\begin{center}


\centerline{ \xy (12.5,5.5);(16,2) **@{-}, (12,2);(15.5,5.5) **@{-},
(17.5,7.5);(22.3,6) **\crv{(19,9)&(21,9)}, (16,2);(22,2)
**\crv{(18,0)&(20,0)}, (22.3,6);(22,2) **\crv{(23.5,4)}, (10.5,7.5);(7,11) **@{-}, (7,6.5);(21,6.5) **@{-}, (23,6.5);(24,6.5) **@{-}
, (7,0);(12,2) **\crv{(10,0)},
 (40,7.5) *{(a)}, (40,0.5) *{(b)},
(35,5);(45,5) **@{-} ?>*\dir{>}, (35,3);(45,3) **@{-} ?<*\dir{<},
 (-5,4)*{\Omega_{9}^* :}, (73,4)*{},
(14,6);(14,2) **@{-}, (14.1,6);(14.1,2) **@{-}, (13.9,6);(13.9,2)
**@{-}, (14.2,6);(14.2,2) **@{-}, (13.8,6);(13.8,2) **@{-}, 
(57,6.5);(60,6.5) **@{-},(74,6.5);(64,6.5) **@{-},(57,11);(57,0) **\crv{(68,4),(68,7)},
\endxy}

\vskip.3cm

\centerline{ \xy (12.5,5.5);(16,2) **@{-}, (12,2);(15.5,5.5) **@{-},
(17.5,7.5);(22.3,6) **\crv{(19,9)&(21,9)}, (16,2);(22,2)
**\crv{(18,0)&(20,0)}, (22.3,6);(22,2) **\crv{(23.5,4)}, (10.5,7.5);(7,11) **@{-}, (7,6.5);(21,6.5) **@{-}, (23,6.5);(24,6.5) **@{-}
, (7,0);(12,2) **\crv{(10,0)}, 
 (40,7.5) *{(a)}, (40,0.5) *{(b)},
(35,5);(45,5) **@{-} ?>*\dir{>}, (35,3);(45,3) **@{-} ?<*\dir{<},
 (-5,4)*{\Omega_{9}'^* :}, (73,4)*{},
 (12,3.9);(16,3.9) **@{-},
**@{-}, (12,4);(16,4) **@{-}, (12,4.1);(16,4.1) **@{-},
(57,6.5);(60,6.5) **@{-},(74,6.5);(64,6.5) **@{-},(57,11);(57,0) **\crv{(68,4),(68,7)},(69.5,6.5) *{>},
(64,4) *{p},(10,2.3) *{l^-},
\endxy}

\caption{Moves $\Omega_9^*$ and $\Omega_9'^*$}\label{fig-omoves}
\end{center}
\end{figure}

\begin{figure}[ht]
\begin{center}
\resizebox{0.65\textwidth}{!}{%
  \includegraphics[width=3.5cm]{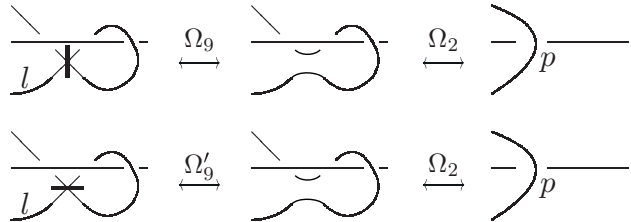}}
\caption{Moves $\Omega_9^*$, $\Omega_9'^*$ are equivalent to $\Omega_9$, $\Omega_9'$.}\label{fig-m9a}
\end{center}
\end{figure}

We conclude the paper by proposing a question. 


\begin{que}
Suppose that $D$ and $D'$ are marked graph diagrams presenting ambient isotopic immersed surface-links.  Is $D$ stably equivalent to $D'$? 
\end{que} 



\begin{thebibliography}{9}





\bibitem{CKS2004} J. S. Carter, S. Kamada and M. Saito, Surfaces in $4$-space, Springer-Verlag, Berlin Heidelberg New York, 2004. 

\bibitem{CS1} J. S. Carter and M. Saito, Knotted surfaces and their diagrams, Mathematical Surveys and Monographs, {\bf 55}, American Mathematical Society, Providence, RI, 1998.






\bibitem{KamBook2017}
S. Kamada, Surface-knots in $4$-space, Springer Monographs in Mathematics, Springer, 2017. 

\bibitem{KamKawamu}
S. Kamada and K. Kawamura, Ribbon-clasp surface-links and normal forms of singular surface-links, 
\textit{Topology Appl}, to appear, ArXiv: 1602.07855v1.   


\bibitem{KSS}
A. Kawauchi, T. Shibuya, S. Suzuki, Descriptions on surfaces
in four-space, I; Normal forms, \textit{Math. Sem. Notes Kobe Univ.} {10} (1982), 75--125.

\bibitem{Kaw}
A. Kawauchi, \textit{A survey of knot theory}, Birkh\"auser, 1996.  

\bibitem{KeKU}
C. Kearton and V. Kurlin, All 2-dimensional links in 4-space live inside a universal 3-dimensional polyhedron, 
\textit{Algebr. Geom. Topol.} {8} (2008), no. 3,  1223--1247. 


\bibitem{Lo}
S. J. Lomonaco, Jr., The homotopy groups of knots I. How to compute the algebraic 2-type, 
\textit{Pacific J. Math.} {95} (1981), 349--390. 





\bibitem{Sw}
F. J. Swenton, On a calculus for 2-knots and surfaces in 4-space, 
\textit{J. Knot Theory Ramifications} 10 (2001), 1133--1141.  

\bibitem{Yo}
K. Yoshikawa, An enumeration of surfaces in four-space,
 \textit{Osaka J. Math.} {\bf 31} (1994), 497--522.  


\end{thebibliography}
\end{document}